\documentclass[a4paper,leqno,12pt, reqno]{amsart}

\usepackage{amsmath}
\usepackage{amssymb}
\usepackage{graphicx}
\usepackage{tikz-cd}
\usepackage{hyperref}
\usepackage[all]{xy}

\usepackage{tikz}

\usepackage{xspace}
\usepackage{bm}
\usepackage{amsmath}
\usepackage{amstext}
\usepackage{amsfonts}
\usepackage[mathscr]{euscript}
\usepackage{amscd}
\usepackage{latexsym}
\usepackage{amssymb}
\usepackage{enumerate}
\usepackage{xcolor}
\usepackage{mathtools}

\theoremstyle{plain}
\swapnumbers

\newcommand{\secat}{\operatorname{secat}}

\newcommand{\TC}{\operatorname{\mathsf TC}}
\newcommand{\SC}{\operatorname{\mathsf SC}}
\newcommand{\CC}{\operatorname{\mathsf CC}}

\newcommand{\tcn}{\TC_{n}}
\newcommand{\tcns}{\TC^{\Sigma}_{n}}
\newcommand{\tcs}{\TC^{\Sigma}}
\newcommand{\scs}{\SC^{\Sigma}}
\newcommand{\scn}{\SC_{n}}
\newcommand{\scns}{\SC^{\Sigma}_{n}}
\newcommand{\scnsr}{\SC^{\Sigma, r}_{n}}
\newcommand{\ccs}{\CC^{\Sigma}}
\newcommand{\ccn}{\CC_{n}}
\newcommand{\ccns}{\CC^{\Sigma}_{n}}
\newcommand{\ccnsr}{\CC^{\Sigma, r}_{n}}

\newcommand{\sd}{\operatorname{sd}}
\newcommand{\st}{\operatorname{st}}

\newcommand{\Id}{\operatorname{Id}}

\newcommand{\KK}{\mathcal{K}}
\newcommand{\XX}{\mathcal{X}}

\newcommand{\zz}{\mathbb{Z}_2}

\newcommand{\bx}{\mathbf{x}}
\newcommand{\bX}{\mathbf{X}}

\newtheorem{theorem}{Theorem}[section]
\newtheorem{prop}[theorem]{Proposition}
\newtheorem{lemma}[theorem]{Lemma}
\newtheorem{corollary}[theorem]{Corollary}

\newtheorem*{thma}{Theorem A}
\newtheorem*{thmb}{Theorem B}

\newenvironment{mysubsection}[2][]
{\begin{subsec}\begin{upshape}\begin{bfseries}{#2.}
			\end{bfseries}{#1}}
		{\end{upshape}\end{subsec}}

\theoremstyle{definition}
\newtheorem{definition}[theorem]{Definition}
\newtheorem{example}[theorem]{Example}

\newtheorem{ack}[theorem]{Acknowledgements}
\newtheorem{subsec}[theorem]{}

\theoremstyle{remark}
\newtheorem{remark}[theorem]{Remark}

\begin{document}

\title[Simplicial and combinatorial versions of higher symmetric complexity] {Simplicial and combinatorial versions of higher symmetric topological complexity}

\author{Amit Kumar Paul}
\author{Debasis Sen}
	
\address{Department of Mathematics and Statistics\\
Indian Institute of Technology, Kanpur\\ Uttar Pradesh 208016\\India}

\email{kamitp@iitk.ac.in/amitkrpaul23@gmail.com}
	
\address{Department of Mathematics and Statistics\\
Indian Institute of Technology, Kanpur\\ Uttar Pradesh 208016\\India}

\email{debasis@iitk.ac.in}

\date{\today}
	
	
\begin{abstract}
In this paper, we introduce higher symmetric simplicial complexity $\scns(K)$ of a simplicial complex $K$ and higher symmetric combinatorial complexity $\ccns(P)$ of a finite poset $P$. These are simplicial and combinatorial approaches to symmetric motion planning of Basabe - Gonz\'{a}lez - Rudyak - Tamaki. We prove that the symmetric simplicial complexity $\scns(K)$ is equal to symmetric topological complexity $\tcns(|K|)$ of the geometric realization of $K$ and the symmetric combinatorial complexity $\ccns(P)$ is equal to symmetric topological complexity $\tcns(|\KK(P)|)$ of the geometric realization of the order complex of $P$.
\end{abstract}
\maketitle

\begin{section}{Introduction}
		
The \emph{topological complexity} of a path connected space $X$ was introduced by Farber (cf. \cite{Far03}).   It is a homotopy invariant measure of the complexity to construct a motion-planning algorithm on the space $X$. Let $I = [0,1]$ and $X^I $ denotes the free path space. Consider the fibration,
\begin{equation}\label{p}
p: X^I \rightarrow X\times X,~~ \gamma \mapsto (\gamma(0), \gamma(1)).
\end{equation}
\noindent Then the topological complexity of $X$ is defined to be the least positive integer $k$ such that there exists an open cover $\{U_1, \cdots, U_k\}$ of $X\times X$ with continuous section of $p$ over each $U_i$ (i.e. a continuous map $s_i : U_i \rightarrow X^I$ satisfying $p \circ s_i = \Id_{U_i}$ for $i = 1, 2, \cdots, k$). It is denoted by $\TC(X)$. Generalizing the idea, Rudyak defined higher topological complexity (cf. \cite{Rud10}). He introduced \emph{n-th topological complexity } $\TC_n(X), \ n \geq 2$  such that $\TC_2(X) = \TC(X)$. We recall the definition of higher topological complexity in the next section. Farber and Grant introduced symmetric topological complexity $\TC^S(X)$ of a path connected space $X$ with the idea that the motion planning between $(p,q)$ and $(q,p)$ are inverses of each other (cf. \cite{FarG07}). They observed that $\TC^S(X)$ is not a homotopy invariant. To rectify it, Basabe, Gonz\'{a}lez, Rudyak, and Tamaki defined another version of symmetric topological complexity $\tcs(X)$ and generalized it for higher version $\tcns(X)$ (see \cite{BasGRT14}). The $\tcns(X)$ is homotopy invariant and naturally defined. In (\cite{BasGRT14}) the authors proved that this two invariants differ by at most $1$ for $n=2$.
		
A simplicial approach of topological complexity, \emph{simplicial complexity}, was introduced by Gonz\'{a}lez (cf. \cite{Gon18}).
It is proved there that for a finite simplicial complex $K$, the simplicial complexity $\SC(K)$ and the topological complexity of its geometric realization $\TC(|K|)$ are equal. Higher analogue of this result was obtained in \cite{Pau19}. In this paper we introduce \emph{symmetric simplicial complexity} $\scs(K)$ and its higher version $\scns(K)$. We show that they are equal to symmetric topological complexity of the geometric realization of $K$. This result is a symmetric version of \cite[Theorem 3.5]{Pau19}. 
\begin{thma}
For a finite simplicial complex $K$, $\scns(K) = \tcns(|K|)$, for any $n \geq 2$.
\end{thma}
\noindent (See Theorem \ref{thm sym sc sym tc equal})
		
A combinatorial analogue of topological complexity was introduced by Tanaka for a finite space (or equivalently for a finite poset, see Section \ref{finiteposetand finitespace}) (\cite{Tan18}). It is proved there that the combinatorial complexity $\text{CC}^\infty(P)$ of a finite space $P$ represent the topological complexity $\TC(|\mathcal{K}(P)|)$, where $|\mathcal{K}(P)|$ denote the realization of the order complex of $P$ (see Section \ref{simplicialcomplex}). This was generalized to higher combinatorial complexity in \cite{Pau19}. Moreover Tanaka introduced symmetric combinatorial complexities $\CC^S(P)$ and $\CC^{\Sigma}(P)$ for a finite space $P$, where the first one is an analogue of $\TC^S$ and the second one is of  $\tcs$ (cf. \cite{Tan19}). He proved that the symmetric combinatorial complexities of $P$ and symmetric topological complexities of $|\mathcal{K}(P)|$ are same: $\CC^S(P) = \TC^S(|\mathcal{K}(P)|)$ and $\CC^{\Sigma}(P) = \tcs(|\mathcal{K}(P)|)$. Here we introduce \emph{higher symmetric combinatorial complexity} $\ccns(P)$ for a finite space $P$ and generalized the second result. This can also be described as a symmetric analogue of \cite[Theorem 5.6]{Pau19}. 
\begin{thmb}
For a finite space $P$, $\ccns(P) = \scns(\KK(P)) = \tcns(|\KK(P)|)$, $n\geq2$. 
\end{thmb}
\noindent (See Theorem \ref{thm sym cc sym tc equal})
\begin{mysubsection}{Organization} The organization of the rest of the paper is as follows. In Section 2, we recall basic ideas of topological complexity and symmetric topological complexity. In Section 3, we  introduce symmetric simplicial complexity of a simplicial complex $K$ and prove Theorem A. In Section 4, define higher symmetric combinatorial complexity of a finite space $P$ and prove Theorem B. 
\end{mysubsection}
\ack{The first author would like to thank IIT Kanpur for PhD fellowship and the second author would like to thank SERB for project grant MTR/2020/000343. }
\end{section}

\begin{section}{Symmetric topological complexity}
		
In this section we first recall the definition of topological complexity and  its higher versions. Then we review basic concepts of symmetric topological complexity.  For details we refer to \cite{Far03, Far08, Rud10, FarG07, BasGRT14}.
\begin{mysubsection}{Topological complexity}
Let $q : E \rightarrow B$ be a fibration. The \emph{sectional category} of $q$ is the minimum integer $k$ such that $B$ can be cover by $k$ open subsets, $U_1 \cup U_2 \cup \cdots \cup U_k = B$ and on each $U_i$ there is a section $s_i : U_i \to E$ of $q$. It is denoted by
$\secat(q)$. If no such $k$ exists then we say $\secat(q) =\infty$. Then topological complexity of $X$  can be described as $\TC(X) = \secat(p)$, where $p$ is the fibration of Equation \ref{p}.
			
Suppose $I_n, n \geq 2,$ denote the wedge of $n$ intervals $[0, 1] \times \{j\}$, where $(0, j), j\in \{1, 2, \cdots, n\}$ are identified. We denote $j$-th interval by $[0 ,1]_{j}$ and the parameter in $[0 ,1]_{j}$ by $(t,j) = t_j$. For any two points $t^1_j, t^2_j \in [0, 1]_j$ we write $t^2_j \geq t^1_j$ if and only if $t^2 \geq t^1$. Consider the mapping space $X^{I_n}$ and the fibration
\begin{equation}\label{e}
e_n : X^{I_n} \rightarrow X^n ,~~~\alpha \mapsto  (\alpha(1_1), \alpha(1_2),\cdots,\alpha(1_n)).
\end{equation}
\noindent Then the \emph{n-th topological complexity} of $X$ is defined to be $\TC_n(X)  := \secat(e_n)$. It can be defined alternatively as  $\TC_n(X) = \secat(e'_n)$, where 
\begin{equation}\label{e'} e'_n : X^{I} \rightarrow X^n ,  ~~~~e'_n(\alpha) = (\alpha(0), \alpha(\frac{1}{n-1}), \alpha(\frac{2}{n-1}),\cdots,\alpha(1)).
\end{equation}
\noindent This is because $e_n$ and $e_n'$ are both fibrational replacement of the diagonal map $X\to X^n$. Clearly $\TC_2(X) = \TC(X)$.
 Also it is known that $\{\TC_n(X)\}$ is a non-decreasing sequence. If a space $Y$ is homotopy equivalent to $X$, then $\TC_n (Y) = \TC_n (X)$ for any $n \geq 2$. Consequently, $X$ is contractible  if and only if $ \TC_n(X) = 1$ for any $n \geq 2$.
\end{mysubsection}

\begin{mysubsection}{Symmetric topological complexity}
To define symmetric topological complexity we need the notion of  \emph{equivariant sectional category}. Let $G$ be a finite group. A topological space $E$ with an action of a group $G$ is called a $G$-space. A subset $U\subset E$ is called \emph{$G$-invariant} if $gU \subseteq U$ for all $g \in G$. Consider another $G$-space $B$. A fibration $q : E \to B$ is called $G$-fibration if $q$ is a $G$-map, i.e $q(gx) = gq(x)$ for all $x \in E$ and $g \in G$. The \emph{equivariant sectional category} of a $G$-fibration $q$ is the minimum number $k$ such that $B$ can be cover by $G$-invariant open subsets, $U_1 \cup U_2 \cup \cdots \cup U_k = B$ and on each $U_i$ there is a $G$-section (a section which is $G$-map) of $q$. It is denoted by $\secat_G(q)$. If no such $k$ exists then we say $\secat_G(q) =\infty$.
					
Farber and Grant (\cite{FarG07}) introduced \emph{symmetric topological complexity} $\TC^S(X)$ for a path connected space $X$. For symmetric motion planning they consider a section $s : X \times X \to X^I$ (not necessarily continuous) of $p$ in Equation (\ref{p}), such a way that $s(x, x)(t) = x$ and $s(x, y)(t) = s(y, x)(1-t)$, for all $x, y\in X$ and $t \in I$. For this they take the subspaces $P'X = \{\alpha \in X^I; ~\alpha(0)\neq \alpha(1)\}$ and $F(X;2) = \{(x,y)\in X \times X; ~x \neq y \}$ and restrict the fibration $p$ on $P'X$, $$p' : P'X \to F(X;2).$$ Note that $\zz$-acts freely on both $P'X$ and $F(X;2)$ by $\alpha \to \alpha^{-1}$ and $(x, y) \to (y, x)$ respectively and $p'$ is $\zz$-map. Also, $P'X = p^{-1}(F(X;2))$. So $p'$ is a $\zz$-fibration. The symmetric topological complexity $\TC^S(X)$ is defined to be $\secat_{\zz}(p')+1$. The extra one comes to define motion planning on the diagonal of $X \times X$. Farber and Grant noticed that $\TC^S(X)$ is not homotopy invariant. To overcome this problem, Basabe, Gonz\'{a}lez, Rudyak, and Tamaki introduced another version of symmetric topological complexity $\tcs(X)$ of $X$, which is easier to handle and natural (cf. \cite{BasGRT14}). They also generalized it to
higher version $\tcns(X)$. In this paper we use Basabe, Gonz\'{a}lez, Rudyak, and Tamaki's definition of symmetric topological complexity. Let us recall the definition.
			
Consider the symmetric group $\Sigma_n$ of permutations of $n$ symbols. For any $g, g'\in \Sigma_n, j\in\{1, 2, \cdots, n\}$ we have $g'\circ g(j)=g'(g(j))$. Take the action (left) of $\Sigma_n$ on $X^n$ by permuting elements: that is $g(x_1, x_2, \cdots x_n) = (x_{g(1)},
x_{g(2)}, \cdots, x_{g(n)})$. Also, $\Sigma_n$ acts on ${I_n}$ by $gt_j=t_{g(j)}$, for $t\in[0, 1]$ and $j\in\{1, 2, \cdots, n\}$. This induces an action of $\Sigma_n$ on $X^{I_n}$ given by $g\alpha(t_j)=\alpha(t_{g(j)})$. So $X^n$ and $X^{I_n}$ are $\Sigma_n$-spaces and the fibration $e_n$ in Equation \ref{e} is a $\Sigma_n$-fibration.
\begin{definition}[\cite{BasGRT14}]
The \emph{symmetric topological complexity} $\tcns(X)$ of $X$ is defined to be:
$$\tcns(X) := \secat_{\Sigma_n}(e_n).$$
\end{definition}
\noindent Here a $\Sigma_n$-invariant subset of $X^n$ is called \emph{symmetric subset} and a section on a symmetric subset of $X^n$ is called a \emph{symmetric section}. From the definition it is clear that $\tcns(X) \geq \tcn(X)$. In (\cite{BasGRT14}), authors proved that $\tcns(X)$ is homotopy invariant, for any $n \geq 2$.

\begin{lemma}
Let $X$ and $Y$ be any two spaces. Then $n$ maps $f_1, f_2, \cdots, f_n : X \to Y$ are in same homotopy class if and only if there is a map $H : X \times I_n \to Y$ such that $H(x, 1_j) = f_j(x)$ for any $1\leq j \leq n$.					
\end{lemma}
\begin{proof}
Let us assume that the maps $f_1, f_2, \cdots, f_n : X \to Y$ are in same homotopy class. Then there are homotopy maps $H_j : X \times I \to Y$ such that $H_j(x, 0) = f_1(x)$ and $H_j(x, 1)= f_j(x)$ for $1\leq j \leq n$. That means $H_j$ is a homotopy from $f_1$ to $f_j$. Since $H_j$ assume the value at $(x, 0)$ for all $j$, so we can define a map $H : X \times I_n \to Y$ with $H(x, 1_j) = f_j(x)$ for $1\leq j \leq n$.	
	
Conversely, if we restrict the map $H$ on union of $[0, 1]_i$ and $[0, 1]_j$ in $I_n$, it gives homotopy between $f_i$ and $f_j$ for $1\leq i\neq j\leq n$.	
\end{proof}

In the above case we call $H$ a \emph{homotopy between} $f_1, f_2, \cdots, f_n$.
\begin{definition}\label{dsymho}
 Assume that $\bX$ is a $\Sigma_n$-space and $Y$ any topological space. Then maps $f_1, f_2, \cdots, f_n: \bX \to Y$, satisfying  $f_j(gx)=f_{g(j)}(x)$, are called \emph{symmetrically homotopic} if  there is a homotopy $H : \bX \times I_n \to Y$ between them satisfying $H(gx, t_j) = H(x, t_{g(j)})$ $g\in \Sigma_n, ~x\in \bX, ~t\in[0, 1], ~1\leq j \leq n$. In this case we call $H$ a \emph{symmetric homotopy} between $f_1, f_2, \cdots, f_n$. Note that the homotopy satisfies the relation $H(gx, 0) = H(x, 0), ~~x \in \bX, ~g \in \Sigma_n$.
\end{definition}
Let $X$ be a topological space and $A$ be a $\Sigma_n$-invariant subset of $X^n$. Define $p_j: A \to X$ as the composition $A
\hookrightarrow X^n \to X$, where the first map is the inclusion and the second map is the projection onto the $j$-th factor. Then clearly $p_j(gx)=p_{g(j)}(x),~x \in A, g\in \Sigma_n$. We will use the following lemma to define symmetric simplicial complexity.

\begin{lemma}\label{lemma homotopy}
With notations as above, the maps $\{p_j:~ 1\leq j\leq n\}$ are symmetrically homotopic if and only if $e_n : X^{I_n} \to X^n$ admits a symmetric section on $A$.
\end{lemma}

\begin{proof}
A symmetric section $s : A \to X^{I_n}$ of $e_n$ satisfies, $s(gx)(t_j) = s(x)(t_{g_(j)}).$ A symmetric homotopy $H : A \times I_n \to X$ between $p_1, p_2, \cdots, p_n$ satisfies $H(gx, t_j) = H(x, t_{g(j)}).$ Hence if we set $$s(x)(t_j) = H(x, t_j),~~x\in A, t_j \in I_n,$$ then the existence of one of $s$ and $H$ implies the existence of the other.
\end{proof}
			
\begin{remark}\label{remark sym tc definition} 
In view of Lemma \ref{lemma homotopy}, the symmetric topological complexity $\tcns(X)$ can be described as the minimum integer $k$ such that $X^n$ can be covered by $\Sigma_n$-invariant open subsets, $U_1 \cup U_2 \cup \cdots \cup U_k = X^n$ and on each $U_i$ composition maps $p_1, p_2, \cdots, p_n : U_i \hookrightarrow X^n \to X$ are symmetrically homotopic. 
\end{remark}
 The following proposition is a simple equivariant analogue of \cite[Proposition 2.2]{Pau19}.
			
\begin{prop}\label{prop locally complact}
Let $X$ be an ENR. Then $\tcns(X) = k$, where $k$ is the minimal integer such that there exist a $\Sigma_n$-equivariant section $s : X^n
\rightarrow X^{I_n}$ (which is not necessarily continuous) of the fibration $e_n$ and a splitting $ G_1 \sqcup G_2 \sqcup \cdots \sqcup G_r = X^n$, where each $G_i$ is locally compact and $\Sigma_n$-invariant subset of $X^n$ and each restriction $s_{|G_i} : G_i \rightarrow X^{I_n}$ is continuous for $i = 1, 2,\cdots,r$.
\end{prop}
			
\begin{proof}
The proof is similar to \cite[Proposition 2.2]{Pau19}, only here the sections are $\Sigma_n$-map and the subsets are $\Sigma_n$-invariant.
\end{proof}
			
\noindent Combining Remark \ref{remark sym tc definition} and Proposition \ref{prop locally complact}, we have the following.
			
\begin{corollary}\label{prop homotopy}
Let $X$ be an ENR. Then $\tcns(X)$ is the minimal integer $k$ such that there exist a splitting $ G_1 \sqcup G_2 \sqcup \cdots \sqcup G_r = X^n$ with each $G_i$ is locally compact and $\Sigma_n$-invariant subset of $X^n$ and on each $G_i$ the composition of inclusion with the projections maps $p_1, p_2, \cdots, p_n : G_i \hookrightarrow X^n \to X$ are symmetrically homotopic.
\end{corollary}
\end{mysubsection}
\end{section}

\begin{section}{Symmetric simplicial complexity}
Gonz\'{a}lez introduced the notion of simplicial complexity $\SC(K)$ for simplicial complex $K$ (cf. \cite{Gon18}). This is based on contiguity of simplicial maps. It is proved in that simplicial complexity $\SC(K)$  is equal to the topological complexity $\TC(|K|)$ of the geometric realization of $K$, for a finite simplicial complex $K$. This has been generalized for higher simplicial complexity (cf. \cite{Pau19}). In this section we first recall some basic ideas on simplicial complexes and simplicial complexity. After that we introduce symmetric simplicial complexity $\scs(K)$ and its higher version $\scns(K)$. Finally we show that $\scns(K) = \tcns(|K|)$.
\begin{mysubsection}{Simplicial complexes}\label{simplicialcomplex}
We begin by recalling some basic ideas of simplicial complexes (\cite{Bre72, EilS52, Spa66}).  A \emph{simplicial complex} $K$ consists of of a set $V(K)$, called \emph{vertices} and a set $S(K)$ of finite nonempty subsets of of $V(K)$, called \emph{simplexes} such that,
\begin{enumerate}[(a)]
\item Singleton subsets of $V(K)$ is a simplex. 
\item Any non empty subset of a simplex is a simplex.
\end{enumerate} 

 We say $K$ is a finite simplicial complex if the set $V(K)$ is finite. A set $\sigma \in S(K)$ with $q+1$ elements is called a $q$-\emph{simplex} and if $\sigma' \subset \sigma$ then $\sigma'$ is called a \emph{face} of $\sigma$. A \emph{simplicial map} $\phi : K \to L$ is a function from the vertices of $K$ to the vertices of $L$ such that for any simplex $\sigma = \{v_0, v_1, \cdots, v_q\}$ in $K$ the image $\phi(\sigma) = \{\phi(v_0), \phi(v_1), \cdots, \phi(v_q)\}$ is a simplex in $L$ (possibly of lower dimension).  For a poset $P$ the \emph{order complex} $\mathcal{K}(P)$ is the simplicial complex whose vertex set is $P$ and  simplexes are totally ordered finite subsets of $P$.  The simplex set $S(K)$ is naturally a poset with inclusion of faces, called \emph{face poset} and we denote it by $\XX(K)$. The order complex of $\XX(K)$ is called the \emph{barycentric subdivision} $\sd(K)$ of $K$.  Thus the set of vertices of $\sd(K)$ is the set $S(K)$ and a $q$-simplex of $\sd(K)$ is a chain $\sigma_0 \subsetneq \sigma_1 \subsetneq\cdots \subsetneq \sigma_q$ of face inclusions of simplexes of $K$.  For any simplicial complex $K$ the \emph{geometric realization} $|K|$ is the set of all functions $\alpha : V(K) \to I=[0, 1]$ such that: (i) for any $\alpha$, the set $\{v\in V(K); \alpha (v)\neq 0\}\in S(K)$, (ii) for any $\alpha, \sum_{v\in V(K)}\alpha(v)=1$.  Then the linear map $|\sd(K)|  \to |K|$ takes each vertex of $\sd(K)$ to the corresponding point of $|K|$ which is a homeomorphism.  For any vertex $v\in V(K)$ the \emph{open star} of $v$ denoted by $\st(v)$ and defined as $\st(v)= \{\alpha \in |K|; \alpha(v)\neq 0\}$. Recall that a vertex map $\phi: V(K) \to V(L)$ is a simplicial approximation of a continuous map $f: |K|\to |L|$ if and only if $f(\st(v)) \subset \st(\phi(v))$ for all $v \in V(K).$

The categorical product of simplicial complexes do not possess the desired property that $|K\times L| = |K| \times |L|.$ To overcome this we need the notion of ordered simplicial complex (cf. \cite{EilS52, PorSZ17}).
\begin{definition}
An \emph{ordered simplicial complex} $K$ is a simplicial complex $K$ together with a partial order on its set of vertices, restricting to a total order on each of its simplices.  The simplexes are denoted by ordered sets $\{v_0 \leq v_1 \leq \cdots \leq v_q\}.$
\end{definition}
\begin{example}\label{exosc}
\begin{enumerate}[(a)]

\item Every simplicial complex $K$ can be thought as an ordered simplicial complex by selecting a total order on its vertex set.
\item The order complex of a poset is an ordered simplicial complex. 
\item In particular, for any simplicial complex $K$, the subdivision $\sd(K)$ is an ordered simplicial complex structure with respect to inclusion of faces ordering.
\end{enumerate}
\end{example}

\noindent In general, in an ordered simplicial complex vertices of each simplex is totally ordered but a totally ordered finite subset of vertices may not be a simplex.  
 
 \begin{definition}\label{dscprod}
 The \emph{cartesian product} $K\times L$ of two ordered simplicial complexes  $K$ and $L$ is also an ordered simplicial complex whose vertex set is $V(K) \times V(L)$ with partial order given by $(u_1, v_1) \leq (u_2, v_2)$ if and only if $u_1 \leq u_2$ and $v_1 \leq v_2$. An ordered set $\{(u_0, v_0)\leq (u_1, v_1) \leq \cdots \leq (u_q, v_q)\}$ is a $q$-simplex in $K\times L$ if $\{ u_0\leq u_1 \leq \cdots \leq u_q\}$ and $\{v_0\leq v_1 \leq \cdots \leq v_q\} $ are simplexes of $K$ and $L$ respectively. 
 
 Then the projection maps $p_1: K\times L \to K$ and $p_2: K \times L \to L$ induces homeomorphism $|p_1|\times |p_2|: |K\times L| \to |K|\times |L|$.  In particular $|K^n| = |K|^n.$
\end{definition}

The notion of homotopy of continuous maps is replaced by \emph{contiguity} of simplicial maps (see \cite{Spa66}). For a positive integer $c$, two simplicial maps $\phi , \phi' : K \rightarrow L$ are called $c$-\emph{contiguous} if there is a sequence of simplicial maps $\phi = \phi^0, \phi^1, \phi^2 \cdots, \phi^c = \phi' : K \rightarrow L$, such that $\phi^{i-1}(\sigma) \cup \phi^{i}(\sigma)$ is a simplex of $L$ for each simplex $\sigma$ of $K$ and $i \in \{1, 2,\cdots,c\}$.  We write $\phi \sim \phi'$ if  $\phi$ and $\phi'$ are $c$-contiguous for some positive integer $c$. This defines an equivalence relation on the set of simplicial maps $K \rightarrow L$ and the equivalence classes are called \emph{contiguity classes}. If $\phi_1, \phi_2, \cdots, \phi_n : K \to L$ are simplicial maps belonging to the same contiguity class then there is a simplicial map $\phi_* : K \to L$ and a positive integer $c$ such that each $\phi_j$ is $c$-contiguous with $\phi_*$. Simplicial approximations of a same continuous map are $1$-contiguous, i.e, they are in same contiguity class. For  any simplicial complexes $J, K, L, M$, if the simplicial maps $\phi, \phi' : K \to L$ are $c$-contiguous, then for any simplicial maps $\psi : J \to K$, $\theta : L \to M$, the composition simplicial maps $\theta \circ \phi \circ \psi, \theta \circ \phi'\circ \psi : J \to K \to L \to M$ are also $c$-contiguous. 

\begin{mysubsection}{Equivariant simplicial complexes}
Let $G$ be a finite group. A \emph{simplicial $G$-complex} is a simplicial complex $K$ with simplicial $G$-action, that is: the group $G$-acts on the vertex set $V(K)$ such that for a simplex $\sigma = \{v_0, \cdots, v_q\}\in S(K)$ we have $g\sigma:=\{gv_0, \cdots, gv_q\}$ is also a simplex (of same dimension).  In addition, if $K$ is an ordered simplicial complex and $G$ preserves the ordering of
each simplex of $K$, we call $K$ an \emph{ordered simplicial $G$-complex}. A \emph{morphism} between simplicial $G$-complexes $K$ and $L$ is a simplicial map $\phi: K\to L$ such that $\phi$ is $G$-equivariant map on the vertex set (cf.  \cite{Bre72}, \cite{PorSZ17}).
\begin{example}\label{exprod}
\begin{enumerate}[(a)]
\item  Let $K$ be an simplicial complex. Consider it as an ordered simplicial complex (see Example \ref{exosc}). Then $K^n = K\times K \times \cdots \times K$ is also an ordered simplicial complex (see Definition \ref{dscprod}). It is a simplicial $\Sigma_n$-complex with natural permutation action on the vertex set $V(K)^n$: $g(v_1, v_2, \cdots, v_n) = (v_{g(1)}, v_{g(2)}, \cdots, v_{g(n)}) $. This action also preserves the ordering. So $K^n$ is an ordered simplicial $\Sigma_n$-complex.
\item Let $K$ be a simplicial $G$-complex. Then the face poset $\XX(K)$ has a natural order preserving $G$-action. Hence ts order complex  $\KK(\XX(K))$, which is the barycentric subdivision $\sd(K),$ is also a simplicial $G$-complex. Moreover $\sd(K)$ is ordered (see Example \ref{exosc}) and the $G$-action is order preserving.
 \end{enumerate}
 \end{example}

Let $K$ and $L$ be two simplicial $G$-complexes and $f : |K| \to |L|$ be a $G$-map. A simplicial approximation $\phi : \sd^r(K) \to L$ of $f$ is called a \emph{$G$-simplicial approximation} or simply \emph{$G$-approximation} if $\phi$ is $G$-equivariant on the vertex sets.

\begin{lemma}\label{lemma approx of identity}
For an ordered simplicial $G$-complex $K$, there is a $G$-simplicial approximation $\iota : \sd(K) \to K$ of identity $\Id : |K| \to |K|$. 
\end{lemma}

\begin{proof}	
Define a map $$\tau : \XX(K) = V(\sd(K)) \to V(K), ~~~~\{v_0\leq v_1\leq \cdots \leq v_k\} \mapsto v_k,$$ for any $\{v_0\leq v_1\leq \cdots \leq v_k\}\in \XX(K)$. This gives a simplicial map $\KK(\tau) = \iota : \sd(K) \to K$.  Note that the map $\iota$ is a (order preserving) $G$-map on the vertex sets. Since $$\Id(\st\{v_0\leq v_1\leq \cdots \leq v_k\})\subset \st(\{v_k\}),$$ so $\iota$ is a $G$-approximation of identity on $|K|$ (cf. \cite{HarV93}).
\end{proof}
\end{mysubsection}

Following definition is the simplicial analogue of symmetric homotopy (see Definition \ref{dsymho}). 
 
\begin{definition}			
Let $K$ be a simplicial $\Sigma_n$-complex, $L$ be any simplicial complex. Simplicial maps $\phi_1, \phi_2, \cdots, \phi_n : K \to L$, satisfying $\phi_j(gv) = \phi_{g(j)}(v),$ are called \emph{symmetrically contiguous} if  there is a simplicial map $\phi_* : K \to L$ and a positive integer $c$ such that  $\phi_*(gv)=\phi_*(v)$, each $\phi_j$ is $c$-contiguous with $\phi_*$ with intermediate maps $\phi^l_j$ satisfying $\phi^l_j(gv) = \phi^l_{g(j)}(v)$,  $v\in V(K), ~g\in \Sigma_n, ~1 \leq l\leq c , ~1\leq j \leq n$. 
\end{definition}	
	
We need the following result later. 
			
\begin{lemma} (\cite[Lemma 3.5.2]{Spa66},  \cite[Lemma 2.5]{Tan19})\label{lemma contiguity to homotopy}
Simplicial maps $\phi_1, \phi_2, \cdots, \phi_n : K \to L$ in the same contiguity class have homotopic topological realization. Moreover, if $K$ is a simplicial $\Sigma_n$-complex and $\phi_1, \phi_2, \cdots, \phi_n : K \to L$ are symmetrically contiguous then the realizations also symmetrically homotopic.
\end{lemma}

\begin{proof}
Assume that $\phi_1, \phi_2, \cdots, \phi_n : K \to L$ lie in same contiguity class. Then there exist a positive integer $c$ and a simplicial map $\phi_* : K \to L$ such that for each $j\in \{1, 2,\cdots, n\}$, $\phi_*$ is $c$-contiguous with $\phi_j$ via the simplicial maps $\phi_*= \phi^0_j, \phi^1_j, \phi^2_j, \cdots, \phi^c_j = \phi_j : K \to L$. We define $H: |K| \times I_n \to |L|$ by, 
$$H(x, t_j) = c\Big(\frac{l+1}{c}-t\Big)\big(|\phi^l_j|(x)\big)+c\Big(t-\frac{l}{c}\Big)\big(|\phi^{l+1}_j|(x)\big), \ \text{for} \ t_j\in\Big[\frac{l}{c}, \frac{l+1}{c}\Big]_j$$
 where $1\leq j \leq n, l\in\{0, 1, \cdots, c-1\}$, $x\in|K|$ and $t$ denote the value of $t_j$ in $j$th interval. Then $H$ is a homotopy between $|\phi_1|, |\phi_2|, \cdots, |\phi_n|$. Moreover, if the maps $\{\phi_i:~~1\leq i\leq n\}$ are symmetrically contiguous then $H$ is a symmetric homotopy.
\end{proof}
		
 Iterated barycentric subdivisions of $K$ are defined by $\sd^{r+1}(K) := \sd(\sd^r(K)).$ The following proposition is a generalized version of \cite[Theorem 3.5.6]{Spa66},  \cite[Theorem 3.49]{Pra07} and \cite[Lemma 3.13]{Tan19}.

\begin{prop}\label{prop homotopy to contiguity}
Let $K$ be a simplicial $\Sigma_n$-complex and $L$ be any other simplicial complex.  If $n$ maps $f_1, f_2,\cdots,f_n : |K| \rightarrow |L|$ are symmetrically homotopic, then there is $r_0 \in \mathbb{N}$ and simplicial approximations $\phi_1, \phi_2,\cdots,\phi_n : \sd^{r_0}(K) \rightarrow L$ of  $f_1, f_2,\cdots,f_n$ respectively which are symmetrically contiguous.
\end{prop}

\begin{proof}
We prove the proposition in four steps.

\textbf{Step I:} Let $H : |K| \times I_n \to |L|$ be a symmetric homotopy between $f_1, f_2, \cdots, f_n$. Since $|K|$ is compact, there are points $0_1=t_1^0 \leq t_1^1 \leq \cdots \leq t_1^c=1_1$ in the first interval of $I_n$ such that for any $x\in|K|$ the points $H(x, t_1^l)$ and  $H(x, t_1^{l-1})$ belong to open star $\st(w)$ for some vertex $w$ of $L$ and $l\in\{1, 2, \cdots, c\}$. We denote $h_j^l(x) = H(x, t_j^l)$ for each $j$ and $l$. By \cite[Theorem 3.5.6]{Spa66}, there is $r_0 \in \mathbb{N}$ (large enough) and simplicial approximation $\psi_1^l : \sd^{r_0}(K) \rightarrow L$ of $h_1^l$ and $h_1^{l-1}$ for $l\in\{1, 2, \cdots, c\}$. 

\textbf{Step II:} Using $\psi_1^l$ here we construct another simplicial approximation $\phi_1^l$ of $h_1^l$ and $h_1^{l-1}$. Let $S=\big\{(1, j) g  (1, g(j))\in \Sigma_n; \ 1\leq j \leq n, \ g\in\Sigma_n\big\}$ and $G$ be the subgroup of $\Sigma_n$ generated by $S$. Take the induced action of $\Sigma_n$ on $\sd^{r_0}(K)$. Choose and fix an element $v_0$ on each $G$-orbit of $\sd^{r_0}(K)$. Now for each $l$ we define $\phi_1^l : \sd^{r_0}(K) \rightarrow L$ as: for any vertex $v$ of $\sd^{r_0}(K)$, $\phi_1^l(v)=\psi_1^l(v_0)$ where $v_0$ is the chosen point on the $G$-orbit of $v$. We claim that $\phi_1^l$ is a simplicial approximation of both $h_1^l$ and $h_1^{l-1}$. Write $v=g'v_0$ for some $g'\in G$. Without loss of generality we can take $g'\in S$, i.e, $g'=(1, j) g  (1, g(j))$ for some $j$ and $g\in \Sigma_n$. Observe that,
\begin{align*}
h_1^l\Big(\st \big(v\big)\Big) &= h_1^l\Big(\st \big(g'v_0\big)\Big) = h_1^l\Big(\st \big( (1, j) g  (1, g(j))v_0\big)\Big)\\
&=  h_1^l\Big((1, j) g  (1, g(j)) \st (v_0)\Big)  =  h_j^l\Big(g  (1, g(j)) \st (v_0)\Big) \\
&=  h_{g(j)}^l\Big((1, g(j)) \st (v_0)\Big)  = h_{1}^l\Big(\st (v_0)\Big) \subset \st \big( \psi_1^l(v_0)\big)= \st \Big( \phi_1^l \big( v\big)\Big).
\end{align*}
\noindent The inclusion follows since $\psi_1^l$ is approximation of $h_1^l$. Similarly we can show that $h_1^{l-1}\big(\st (v)\big) \subset \st \big( \phi_1^l (v)\big)$. So $\phi_1^l$ is a simplicial approximation of both $h_1^l$ and $h_1^{l-1}$. 

\textbf{Step III:} We now use $\phi_1^l$ to define $\phi_j^l : \sd^{r_0}(K) \rightarrow L$ which is a simplicial approximations of both $h_j^l$ and $h_j^{l-1}.$  For $1\leq j \leq n$ and $1\leq l\leq c$,
$$\phi_j^l : \sd^{r_0}(K) \rightarrow L, ~~\phi_j^l(v) :=\phi_1^l\big((1, j) v\big) \text{ for } v\in V(\sd^{r_0}(K)).$$ 
Note that,
\begin{align*}
h_j^l\Big(\st (v)\Big) \cup h_j^{l-1}\Big(\st (v)\Big)&=  h_1^l\Big((1, j)\st (v)\Big) \cup h_1^{l-1}\Big((1, j)\st (v)\Big) \\
&=  h_1^l\Big(\st \big((1, j)v\big)\Big) \cup h_1^{l-1}\Big(\st \big((1, j)v\big)\Big)  \\
&\subset \st \Big( \phi_1^l\big((1, j)v\big)\Big)= \st \Big( \phi_j^l(v)\Big).
\end{align*}
So $\phi_j^l$ is a simplicial approximation of both $h_j^l$ and $h_j^{l-1}$. In particular  $\phi_j^c$ ($=\phi_j$, say) is a simplicial approximation of $h_j^c=f_j$. Moreover, for any vertex $v$ of  $\sd^{r_0}(K)$ and $g\in \Sigma_n$ we denote $u=(1, g(j))v \Leftrightarrow (1, g(j))u = v$. Then we have,
\begin{align*}
 \phi_j^l(gv) &=  \phi_j^l\big(g  (1, g(j))u\big) = \phi_1^l\big((1, j)  g  (1, g(j))u\big)  = \phi_1^l\big(u\big)  = \phi_1^l\big((1, g(j))v\big)= \phi_{g(j)}^l(v).
\end{align*}
\noindent Therefore the simplicial approximation $\phi_j^l$ of $h_j^l$ and $h_j^{l-1}$ satisfies $\phi_j^l(gv) = \phi_{g(j)}^l(v)$.

\textbf{Step IV:} Now $h_j^0=h^0 : |K| \to |L|$ is an $\Sigma_n$-equivariant map (thought $|L|$ has trivial $\Sigma_n$-action), so by  \cite[Theorem 3.49]{Pra07}, there is an $\Sigma_n$-equivariant simplicial approximation of $h_0$, $\phi_j^0=\phi^0 : \sd^{r_0}(K) \to L$, for $r_0$ large enough. So $\phi^0(gv)=\phi^0(v)$ for any vertex $v$ of $ \sd^{r_0}(K)$ and $g\in \Sigma_n$. We set $\phi_*=\phi^0$. Note that the simplicial maps  $\phi_j^{l}$ and $\phi_j^{l+1}$ are $1$-contiguous (since these are approximations of same map $h_j^l$),  for $l\in\{0, 1, \cdots, c-1\}$ and $j\in\{1, 2, \cdots, n\}$. So the simplicial approximations $\phi_1, \phi_2,\cdots,\phi_n : \sd^{r_0}(K) \rightarrow L$ of  $f_1, f_2,\cdots,f_n$ respectively, are symmetrically contiguous.
\end{proof}
\end{mysubsection}

\begin{mysubsection}{Simplicial complexity} 
Here we recall simplicial complexity $\scn(K)$ of a simplicial complex $K$ from (\cite{Pau19}). Choose a simplicial approximation $\iota_{ K^n}: \sd(K^n) \rightarrow  K^n$ of the identity on $|K^n|=|K|^n$. We denote
 \begin{equation}
 \iota^{r}_{K^n}:\sd^{r}(K^n) \rightarrow  K^n
 \end{equation} 
as the iterated composition
\begin{center}
$ \sd^r(K^n) \xrightarrow {\iota_{\sd^{r-1}(K^n)}} \sd^{r-1}(K^n) \xrightarrow {\iota_{\sd^{r-2}(K^n)}} \cdots   \xrightarrow {\iota_{\sd(K^n)}} \sd(K^n) \xrightarrow {\iota_{K^n}} K^n$
\end{center}
and $p_j\circ \iota^{r}_{K^n}= \pi_{j} : \sd^{r}(K^n) \rightarrow K$  where  $p_j: K^n \to K$  is the $j^{th}$ projection $r\geq 0, ~1\leq j\leq n$. Then $\SC^{r}_n(K)$ is the smallest non-negative integer $k$ such that there exist subcomplexes $\{L_i\}_{i=1}^{k}$ covering $\sd^{r}(K^n)$ and the restrictions $\pi_j : L_i \rightarrow K$, for $j = 1, 2,\cdots,n,$ lie in the same contiguity class, for each $i$. If no such $k$ exists then $\SC^{r}_n(K)=\infty$.
The value $\SC^{r}_n(K)$ independent of the chosen of approximation $\iota^{r}_{K^n}:\sd^{r}(K^n) \rightarrow  K^n $ of the identity on $|K|^n$. It is also shown that $\{\SC_n^r(K)\}_r$ is a decreasing sequence and bounded below by $1$. So we define
 the $n$-th \emph{simplicial complexity} as $\scn(K) :=  \min_{r \geq 0}\{\SC_n^{r}(K)\}.$
 Following theorem relates simplicial complexity and topological complexity.
\begin{theorem}[\cite{Pau19}]
For a finite simplicial complex $K$,  $\SC_n(K) = \TC_{n}(|K|), n \geq 2$.
\end{theorem}
			
\end{mysubsection}
		
\begin{mysubsection}{Symmetric simplicial complexity}
 Let $K$ be a simplicial complex. Then $K^n$ is a simplicial $\Sigma_n$-complex (See Example \ref{exprod}). A subcomplex $L$ of $K^n$ is called symmetric if $gL=L$ for all $g\in \Sigma_n$. In this case, $|L|$ of is $\Sigma_n$-invariant. To define symmetric simplicial complexity, we choose an $\Sigma_n$-approximation $\iota_{K^n} : \sd(K^n) \to K^n$ of identity on $|K^n|$. Such an approximation exists by Lemma \ref{lemma approx of identity}. As in previous case $\pi_{j} : \sd^{r}(K^n) \rightarrow K$ denotes the composition of $p_j\circ \iota^{r}_{K^n}.$ 
\begin{definition}
For a simplicial complex $K$ and integer $r\geq 0$, let $\scnsr(K)$ be the smallest non-negative integer $k$ such that there exist symmetric subcomplexes $\{L_i\}_{i=1}^{k}$ covering $\sd^{r}(K^n)$ and the restrictions $\pi_j|_{L_i} : L_i \rightarrow K$ $1\leq j \leq n,$ are symmetrically contiguous on each $L_i$. If no such $k$ exists then $\scnsr(K)=\infty$.
\end{definition}
\noindent Note that the maps $\pi_j$ depend on the choice of an $\Sigma_n$-approximation of identity. The following lemma shows that the above definition is independent of such a choice. 
\begin{lemma}
The value $\scnsr(K)$ is independent of the chosen of $\Sigma_n$-approximation $\iota_{K^n}^r : \sd^{r}( K^n) \rightarrow K^n$ of the identity on $|K|^n$.
\end{lemma}

\begin{proof}				
Let $r'$ be any number such that $1\leq r' \leq r$. We fix iterated compositions of $\Sigma_n$-approximations $\iota^{r'-1}_{ K^n} : \sd^{r'-1}(K^n) \rightarrow  K^n $ and $\iota^{r-r'}_{\sd^{r'}(K^n)} : \sd^{r}(K^n) \rightarrow  \sd^{r'}(K^n) $ of the identity on $|K|^n$. Now we take two $\Sigma_n$-approximations $\iota_{\sd^{r'-1}( K^n)}, \bar{\iota}_{\sd^{r'-1}( K^n)} : \sd^{r'}(K^n) \rightarrow \sd^{r'-1}( K^n)$ of the identity on $|K|^n$. Let $\pi_{j} , \bar{\pi}_j: \sd^{r}(K^n) \rightarrow K$ be the compositions $p_j \circ \iota^{r'-1}_{ K^n} \circ \iota_{\sd^{r'-1}( K^n)} \circ \iota^{r-r'}_{\sd^{r'}(K^n)}$ and $p_j \circ \iota^{r'-1}_{ K^n} \circ \bar{\iota}_{\sd^{r'-1}( K^n)} \circ \iota^{r-r'}_{\sd^{r'}(K^n)}$ respectively, where $p_j : K^n \to K$ is the $j$-th projection. Let $\pi_j : L \to K$ for $j\in \{1, 2, \cdots, n\}$ be symmetrically contiguous by the contiguity chain $\pi_*=\pi^0_j, \pi^1_j, \cdots, \pi^c_j=\pi_j: L \to K$, on some symmetric subcomplex $L$ of $\sd^{r}(K^n)$. Since $\iota_{\sd^{r'-1}( K^n)}, \bar{\iota}_{\sd^{r'-1}( K^n)} : \sd^{r'}(K^n) \rightarrow \sd^{r'-1}( K^n)$ both are approximation of identity on $|K|^n$, so they are $1$-contiguous and hence $\pi_j, \bar{\pi}_j : L \to K$ are $1$-contiguous. So $\pi_*=\pi^0_j, \pi^1_j, \cdots, \pi^c_j=\pi_j, \bar{\pi}_j: L \to K$ is a contiguity chain on the subcomplex $L$. Now since $\iota^{r'-1}_{ K^n}$, $\iota^{r-r'}_{\sd^{r'}(K^n)}$ and $\bar{\iota}_{\sd^{r'-1}( K^n)}$ are $\Sigma_n$-simplicial maps, for any vertex $v$ of $L$ and any $g\in\Sigma_n$ we have
\begin{align*}
\bar{\pi}_j(gv) &=  p_j \circ \iota^{r'-1}_{ K^n} \circ \bar{\iota}_{\sd^{r'-1}( K^n)} \circ \iota^{r-r'}_{\sd^{r'}(K^n)}\big(gv )  
= p_j  \big(g\big(\iota^{r'-1}_{ K^n} \circ \bar{\iota}_{\sd^{r'-1}( K^n)} \circ \iota^{r-r'}_{\sd^{r'}(K^n)}(v)\big)\big)\\
&=  p_{g(j)}  \big(\iota^{r'-1}_{ K^n} \circ \bar{\iota}_{\sd^{r'-1}( K^n)} \circ \iota^{r-r'}_{\sd^{r'}(K^n)}(v)\big) 
=  p_{g(j)} \circ \iota^{r'-1}_{ K^n} \circ \bar{\iota}_{\sd^{r'-1}( K^n)} \circ \iota^{r-r'}_{\sd^{r'}(K^n)}(v) \\
& = \bar{\pi}_{g(j)}(v).
\end{align*}
So $\bar{\pi}_j : L \to K$ for $j\in \{1, 2, \cdots, n\}$ are symmetrically contiguous by the contiguity chain $\pi_*=\pi_j^0, \pi^1_j, \cdots, \pi^c_j, \pi_j^{c+1}= \bar{\pi}_j : L \to K$, on the subcomplex $L$ of $\sd^{r}(K^n)$. Similarly we can show that if $\bar{\pi}_j$'s are symmetrically contiguous then $\pi_j$'s are so. This is true for any $r'$ between $1$ and $r$. Hence $\scnsr(K)$ is independent of the chosen of $\Sigma_n$-approximation $\iota_{K^n}^r : \sd^{r}( K^n) \rightarrow K^n$ of the identity on $|K|^n$.
\end{proof}
			
 Next we show that $\scnsr(K)$ is bounded below by $\tcns(|K|)$.
 
\begin{lemma}\label{lemma tc sc}
For a simplicial complex $K$, $\tcns(|K|) \leq \scnsr(K)$,  $r\geq 0,~n \geq 2$.
\end{lemma}
			
\begin{proof}				
Let $\scnsr(K) = k$. Let us consider symmetric subcomplexes
$\{L_i\}_{i=1}^{k}$ covering $\sd^{r}(K^n)$ such that the restrictions $\pi_j :L_i \rightarrow K$, for $j = 1, 2,\cdots,n,$ are symmetrically contiguous, for each $i$. Now we apply geometric realization functor. By Lemma \ref{lemma contiguity to homotopy} we get $|\pi_j| : |L_i| \rightarrow |K|$, for $j = 1, 2,\cdots,n,$ are symmetrically homotopic for each $i$. Let $H: |L_i| \times I_n \to K$ be a symmetric homotopy between $\pi_1, \pi_2, \cdots, \pi_n$. We restrict the map $H$ on $|L_i| \times j$-th interval of $I_n$ and denote it by $h_j$. So $h_j : |L_i| \times I \to K$. For each $j$ we define an another map $f_j : |L_i| \times I \to K$ by $f_j(x, t) = t  p_j(x) + (1-t)(|\pi_j|(x))$, where $p_j$ is the composition $|L_i|\hookrightarrow |K|^n \to K$. Now consider the homotopy $F: |L_i| \times I_n \to K$ is defined by: for each $x$, the path $F(x, t_j)$ is the concatenation $(h_j*f_j)(x, t)$,
$$
(h_j*f_j)(x, t) = 
\begin{cases} 
h_j(x, 2t) & \mbox{ if } ~~ t \in [0, \frac{1}{2}]\\
f_j(x, 2t-1) &\mbox{if } ~~ t \in [\frac{1}{2}, 1].
\end{cases}
$$ 

By assumption $h_j(gx, t)= 	h_{g(j)}(x, t)$ and by definition of $f_j, f_j(gx, t)= 	f_{g(j)}(x, t)$ for any $g\in \Sigma_n$ and $F(x, 1_j)=p_j(x)$. So $F$ is a symmetric homotopy and $p_1, p_2, \cdots, p_n$ are symmetrically homotopic on $|L_i|$. Now we set $G_1 = |L_1|$ and $G_i = |L_i| - (|L_1|\cup |L_2| \cup \cdots \cup |L_{i-1}|)$ for $i\geq2$. Then each $G_i$ is locally compact, $\Sigma_n$-invariant and $|K^n| = \sqcup_i G_i$. The restrictions of $p_j$, for $j = 1, 2,\cdots,n,$ on each $G_i$ are also symmetrically homotopic. From Corollary \ref{prop homotopy}, we have $\tcns(|K|) \leq k$. So $\tcns(|K|) \leq \scnsr(K)$, for any $r\geq 0$ and $n \geq 2$.			
\end{proof}

As in the case of simplicial complexity,  now we show that $\{\scnsr(K)\}_r$ is a decreasing sequence.
\begin{lemma}
For a simplicial complex $K$, $\scnsr(K) \geq \SC_{n}^{\Sigma, r+1}(K)$,  $r\geq 0, ~~n \geq 2$.
\end{lemma}
\begin{proof}	
Let $\scnsr(K) = k$	and $\{L_i\}_{i=1}^{k}$ be a symmetric subcomplexes covering $\sd^{r}(K^n)$ such that the restrictions $\pi_j : L_i \rightarrow K$ $1\leq j\leq n,$ are symmetrically contiguous, for each $i$. So there is a contiguity chain $\pi_*=\pi^0_j, \pi^1_j, \cdots, \pi^c_j=\pi_j: L_i \to K$ such that $\pi_j^l(gv) = \pi_{g(j)}^l(v)$ and $\pi_*(gv) = \pi_*(v)$ $v \in V(L_i)$,  $l \in \{ 1, 2,\cdots, c\}$. Take the subcomplexes $\{\sd(L_i)\}_{i=1}^{k}$ which are symmetric and covers $\sd^{r+1}(K^n)$. We choose a $\Sigma_n$-approximation $\iota_{\sd^{r}( K^n)} :\sd^{r+1}(K^n) \rightarrow \sd^{r}( K^n)$ of the identity on $|K|^n$. Clearly, for each $L_i$, the sequence of maps $\pi_* \circ \iota_{\sd^{r}( K^n)} =\pi^0_j \circ \iota_{\sd^{r}( K^n)}, \pi^1_j \circ \iota_{\sd^{r}( K^n)}, \cdots, \pi^c_j \circ \iota_{\sd^{r}( K^n)} = \pi_j \circ \iota_{\sd^{r}( K^n)} : \sd(L_i) \to K$ for $j \in \{ 1, 2,\cdots, n\}$ give a contiguity chain. Since $\iota_{\sd^{r}( K^n)}$ is a $\Sigma_n$-map, for any vertex $v$ of $\sd(L_i)$, $l \in \{ 1, 2,\cdots, c\}$ we have,
$$\pi_* \circ \iota_{\sd^{r}( K^n)}(gv) =\pi_* (g\iota_{\sd^{r}( K^n)}(v)) =\pi_* (\iota_{\sd^{r}( K^n)}(v)) = \pi_* \circ \iota_{\sd^{r}( K^n)}(v), $$
$$\pi_j^l \circ \iota_{\sd^{r}( K^n)}(gv) = \pi_j^l (g \iota_{\sd^{r}( K^n)}(v)) = \pi_{g(j)}^l (\iota_{\sd^{r}( K^n)}(v)) = \pi_{g(j)}^l \circ \iota_{\sd^{r}( K^n)}(v).$$

\noindent This implies that the restrictions $\pi_j \circ \iota_{\sd^{r}( K^n)} : \sd(L_i) \rightarrow K$, for $j \in \{ 1, 2,\cdots, n\}$ are symmetrically contiguous, for each $i$. So  $\SC_{n}^{\Sigma, r+1} \leq k$ and therefore $\scnsr(K) \geq \SC_{n}^{\Sigma, r+1}(K)$. 
\end{proof}
			
 Above Lemma allows us to make the following definition.
\begin{definition}
For a simplicial complex $K$, the \emph{$n$-th symmetric simplicial complexity} or simply \emph{symmetric simplicial complexity} is defined as: $$\scns(K) = \min_{r \geq 0}\{\scnsr(K)\}.$$
\end{definition}
			
 The main theorem of this section is the following.
			
\begin{theorem}\label{thm sym sc sym tc equal}
For a finite simplicial complex $K$, $\scns(K) = \tcns(|K|)$,  $n \geq 2$.
\end{theorem}
			
\begin{proof}				
From Lemma \ref{lemma tc sc}, it is clear that $\scns(K) \geq \tcns(|K|)$. Now we prove the other inequality. Let $\tcns(|K|) = k$. Using Remark \ref{remark sym tc definition} we get an symmetric open cover $\{U_i\}_{i=1}^k$ of $|K|^n$  such that the composition maps $p_1, p_2, \cdots, p_n : U_i \hookrightarrow |K|^n \to |K|$ are symmetrically homotopic for each $i$. Since $K$ is finite, $|K|$ is compact and so is $|K|^n = |K^n|$. Therefore by Lebesgue lemma of compact metric spaces, there exists $\delta >0$ such that any set of diameter less than $\delta$ is contained in one of the open sets $U_i$.  Since, with increasing subdivision, the diameter of simplices goes to $0$, there is a large integer $r\geq 0$ such that realization of each simplex $\sigma$ of $\sd^r(K^n)$ contained one of $U_i$. Let $L_i$ be the subcomplex of $\sd^r(K^n)$ consisting of those simplices whose realization contained in $U_i$. For each $i$, the set is $U_i$ is $\Sigma_n$-invariant  and so for each simplex $\sigma$ of $L_i$, $g|\sigma|$ contained in $U_i$, for any $g\in \Sigma_n$. Therefore $g\sigma$ is a simplex of $L_i$. Thus each subcomplex $L_i$ is symmetric. Also $\{L_i\}_{i=1}^{k}$ covers $\sd^r(K^n)$. Since the maps $p_1, p_2, \cdots, p_n : U_i \hookrightarrow |K|^n \to |K|$ are symmetrically homotopic, their restriction on $|L_i|$, $p_1, p_2, \cdots, p_n : |L_i| \hookrightarrow |K|^n \to |K|$ are also symmetrically homotopic. Now by Proposition \ref{prop homotopy to contiguity} there is positive integer $r_0$ and approximations $\pi_1, \pi_2, \cdots, \pi_n : \sd^{r_0}(L_i)
\hookrightarrow \sd^{r+r_0}(K^n) \rightarrow K$ of $p_1, p_2, \cdots, p_n$ respectively such that they are symmetrically contiguous, for each $i$. Therefore $\scns(K) \leq k$ and hence $\scns(K) = \tcns(|K|)$.
\end{proof}

\end{mysubsection}
\end{section}
	
\begin{section}{Symmetric combinatorial complexity}
		
Tanaka introduced a combinatorial approach of topological complexity (cf. \cite{Tan18}).  The basic idea of  Tanaka's paper is to describe topological complexity by combinatorics of finite posets. He used a combinatorial analogue of the path-space fibration of the Equation \ref{p} to define  \emph{combinatorial complexity} $\CC^0(P)$ of a finite poset $P$, where zero means no barycentric subdivision of $P\times P$  is involved. It is shown that $\CC^0(P) = \TC(P)$ but $\CC^0(P)$ does not capture $\TC(|\KK(P)|)$, where $\KK(P)$ is the \emph{order complex} of $P$. To describe $\TC(|\mathcal{K}(P)|)$ combinatorially, Tanaka used barycentric subdivision of $P\times P$ to define combinatorial complexity $\CC(P)$. Finally it is shown that for a large barycentric subdivision of $P\times P$, $\CC(P) = \TC(|\mathcal{K}(P)|)$. The above idea has been generalized to higher versions and shown that for a finite poset $P$,  $\ccn(P) = \tcn(|\KK(P)|)$ (cf. \cite{Pau19}). Tanaka further defined a combinatorial analog of symmetric topological complexities $\TC^S(X)$ and $\tcs(X)$ for a finite poset $P$ (\cite{Tan19}). We denote this symmetric combinatorial complexities by $\CC^S(P)$ and $\ccs(P)$ respectively. Tanaka there also used barycentric subdivision of $P\times P$ to define these symmetric combinatorial complexities. For a large barycentric subdivision of $P\times P$ the value is stable and he denotes this stable value by $\CC^S(P)$ and $\ccs(P)$ and proved that $\CC^S(P)= \TC^S(|\KK(P)|)$ and $\ccs(P) = \tcs(|\KK(P)|)$. In this section we first recall some basic ideas of finite poset and its connection with finite spaces. Then we recall combinatorial complexity and symmetric combinatorial complexity $\ccs(P)$ of a finite poset $P$. Finally we introduce higher symmetric combinatorial complexity $\ccns(P)$ and prove $\ccns(P) = \tcns(|\KK(P)|)$.
		
\begin{mysubsection}{Finite poset and finite space}\label{finiteposetand finitespace}
We begin by recalling the relation between finite poset and finite space. We refer reader to \cite{Sto66} for this.  Let $P$ be a finite poset. We denote $U_x = \{y: ~~y\leq x\}$. Then $\{U_x; x\in P\}$ generates a $T_0$ topology on the finite set $P$. On the other hand, given a finite $T_0$ topological space $P$, let $U_x$ denotes the intersection of all open sets  containing $x,$ where $x \in P$. Then  we can consider $P$ as a poset, the partial relation on $P$, defined by $x \leq y$ if and only if $U_x \subseteq U_y$. Thus a finite poset is equivalent to a finite $T_0$ space. We will simply write finite space to mean a finite $T_0$ space. From now onwards we assume all our finite spaces are connected. A map between finite spaces is continuous if and only if it preserves the partial order.  Given two finite spaces $P, Q$, we denote by $Q^P$ the space of maps $P\to Q$ with the compact-open topology. This finite $T_0$-space corresponds to the set of order preserving maps $P \to Q$ with the pointwise ordering, i.e. $f \leq g$ if $f (x) \leq g(x)$ for every $x \in P$.

\begin{example}\label{face poset of a simplex is open}
Let $K$ be a finite simplicial complex. Recall that the face poset $\XX(K)$ is the collection of all simplices in $K$ with the partial order of face inclusions. For any simplex $\sigma$ of $K$ we have
$$U_{\sigma} = \{\sigma'; ~ \sigma'\leq \sigma\} = \{\sigma'; ~ \sigma' \text{ is a face of}~ \sigma\} = \XX(\sigma).$$
So $\{\XX(\sigma); ~ \sigma \text{ is a simplex of}~ K\}$ generates a $T_0$ topology on $\XX(K)$.
\end{example}
\begin{remark}\label{face poset of a subcomplex open}
For any subcomplex $L$ of a finite simplicial complex $K$, we have $$\XX(L) = \bigcup_{\sigma\in S(L)}\XX(\sigma).$$ So by Example \ref{face poset of a simplex is open} we can say that $\XX(L)$ is open in $\XX(K)$.
\end{remark}

Let $J_m$ denotes the finite space consisting of $m+1$ points with the zigzag order
\begin{center}
$0 \leq 1 \geq 2 \leq \cdots \geq (\leq)m$.
\end{center}
This finite space is called the \emph{finite fence} with length $m$. It behaves like an interval in the category of finite spaces. An order preserving map $\gamma : J_m \rightarrow P$ is called a combinatorial path or simply a path. Thus a combinatorial path is just a zigzag $\gamma(0) \leq \gamma(1) \geq \gamma(2) \leq \cdots \geq (\leq) \gamma(m)$ of elements of $P$. If $m$ is even, inverse of a path $\gamma : J_m \rightarrow P$ defined by $\gamma^{-1} : J_m \rightarrow P, \gamma^{-1}(i) = \gamma(m-i)$. A connected finite space is always path-connected. Two maps $f, g : P \rightarrow Q$ between two finite spaces are called \emph{combinatorially homotopic} if there exist $m \geq 0$ and a continuous map (or equivalently order preserving map) $H : P \times J_m \rightarrow Q$ such that $H(x, 0) = f(x)$ and $H(x, m) = g(x)$, i.e., there is a fence $f =f_0 \leq f_1 \geq f_2 \leq \cdots f_n =g$.
			
For $n\geq 2$ we denote $J_{n,m}$ be the finite poset of $nm + 1$ points $$\{0, 1_1, 1_2,\cdots ,1_n, 2_1, 2_2,\cdots ,2_n,\cdots ,m_1, m_2,\cdots,m_n\}.$$  The partial ordering on $J_{n,m}$  consists of $n$ finite fences, each of length $m$, as below: $$0 \leq 1_1 \geq 2_1 \leq \cdots \geq (\leq) m_1,$$
$$0 \leq 1_2 \geq 2_2 \leq \cdots \geq (\leq) m_2,$$
$$\cdots $$
$$0 \leq 1_n \geq 2_n \leq \cdots \geq (\leq) m_n.$$ We use the parameter $t_j$ for the $j$-th fence. 
\begin{definition}
Let $f_1, f_2, \cdots, f_n : P \to Q$ be $n$ order preserving maps between two finite spaces. Then we say $f_1, f_2, \cdots, f_n$ are \emph{combinatorially homotopic} if there exists $m \geq 0$ and an order preserving map $H : P \times J_{n, m} \rightarrow Q$ such that $H(x, m_j) = f_j(x)$ for $x\in P$ and $j\in\{1, 2, \cdots, n\}$. Moreover, if $P$ is a $\Sigma_n$-space then we called $f_1, f_2, \cdots, f_n$ are \emph{symmetrically combinatorially homotopic} if the maps satisfies $f_j(gx) = f_{g(j)}(x)$  and the homotopy map satisfy $H(gx, t_j) = H(x, t_{g(j)})$ for any $g\in\Sigma_n, x\in P$ and $j\in\{1, 2, \cdots, n\}$. In this case the homotopy is called \emph{symmetric combinatorial homotopy}.
\end{definition}
The following lemma is a generalization of \cite[Proposition 2.2]{Tan19} .
\begin{lemma}\label{lemma comb homotopy}
Let $P$ be a finite $\Sigma_n$-space and $Q$ be arbitrary finite space. Then any maps $f_1, f_2, \cdots, f_n : P \to Q$ are symmetrically homotopic if and only if they are symmetrically combinatorially homotopic.
\end{lemma}
\begin{proof}
Assume that $f_1, f_2, \cdots, f_n : P \to Q$ are symmetrically homotopic. Then get a symmetric homotopy $H : P \times I_n \to Q$. By homotopy theory of finite spaces, there is a $\Sigma_n$-map $h : J_{n,m} \to I_n$ such that $j$-th fence maps to $j$-th interval and $m_j$ maps to $1_j$, for some $m>0$ and $j\in\{1, 2, \cdots, n\}$. Define $H'$ as the composition map $H' : P \times J_{n, m} \xrightarrow {\Id \times h} P \times I_n \xrightarrow {H} Q$. Then $H'$ is a symmetric combinatorial homotopy between $f_1, f_2, \cdots, f_n$.

Conversely, assume that there is a symmetric combinatorial homotopy  $H' : P \times J_{n, m} \to Q$ between $f_1, f_2, \cdots, f_n$. We define a $\Sigma_n$-map $h' : I_n \to J_{n,m}$ such that each $j$-th interval of $I_n$ maps to $j$-th fence of $J_{n,m}$ given by the equation 
$$
h'(t) =
 \begin{cases}
2k-1 & \mbox{ if } ~~ t = \frac{2k-1}{m},\\
2k &\mbox{if } ~~ \frac{2k-1}{m} < t < \frac{2k+1}{m},
\end{cases}
$$ for $k=0, 1, 2, \cdots$.
Then the composition map $H : P \times I_{n} \xrightarrow {\Id \times h'} P \times J_{n, m} \xrightarrow {H} Q$ gives a symmetric homotopy between $f_1, f_2, \cdots, f_n$.
\end{proof}
Now onwards, we will use symmetrical homotopy to mean any one of the above interpretations. Let us define \emph{barycentric subdivision} of a finite space $P$.
\begin{definition}
The \emph{barycentric subdivision} of a finite space $P$ is defined as the face poset $\mathcal{X}(\mathcal{K}(P))$ of the order complex $\mathcal{K}(P)$ (see  Section \ref{simplicialcomplex}). It is denoted by $\sd(P)$.
\end{definition}

The following Proposition is a symmetric version for $n$ maps of \cite[Lemma 4.10 and Proposition 4.11]{BarM12}.

\begin{prop}\label{homotopy to contiguous in face poset}
Let $P$ be a finite $\Sigma_n$-space and $Q$ be any finite space. Assume that the maps $f_1, f_2, \cdots, f_n : P \to Q$ are symmetrically homotopic. Then the simplicial maps $\KK(f_1), \KK(f_2), \cdots, \KK(f_n) : \KK(P) \to \KK(Q)$ are symmetrically contiguous.
\end{prop}
\begin{proof}
 Without loss of generality we may assume that there is symmetric homotopy $F : J_{n, 1} \times P \to Q$. We write $f_j^0(x)=F(x, 0_j)\leq F(x, 1_j)= f_j(x)$ for all $x\in P$. Then we have $f_j^0(gx)= f_{g(j)}^0(x)$ and  $f_j(gx)= f_{g(j)}(x)$ for any $j$ and any $x\in P, g\in \Sigma_n$. We prove our result in four steps. First we construct sequences of order preserving maps $f_j^0, f_j^1, f_j^2, \cdots, f_j^m=f_j:P \to Q$.
	
\textbf{Step I:} Here we define the maps $f_j^1: P\to Q$. Define $A_j=\{x\in P; f_j^0(x)\neq f_j(x)\}$. Note that $A_{g(j)} = g^{-1}A_j$. So if one of $A_j$'s is empty the all $A_j$ are empty and in this case $f_j^1=f_j$. Assume that $A_1$ is non empty. Then we define $f_j^1: P \to Q$ by,
$$
f_j^1(x) = 
\begin{cases}
f_j(x) & \mbox{ if} ~~ x \in A_j \mbox{ and it is a maximal element of } A_j \\
f_j^0(x) & \mbox{ otherwise.} 
\end{cases}
$$
\noindent We claim that the maps $f_j^1$ are order preserving. Let $x_1, x_2 \in P$ with $x_2\geq x_1$. If none of them is maximal element of $A_j$ then $f_j^1(x_2)\geq f_j^1(x_1)$, since $f_j^0$ is order preserving. If $x_2$ is maximal element of $A_j$ and $x_1$ is any other element of $P$ with $x_2\geq x_1$ (so $x_1$ can not be maximal of $A_j$) then $f_j^1(x_2)=f_j(x_2)\geq f_j(x_1)\geq f_j^0(x_1)= f_j^1(x_1)$. Lastly if $x_1$ is maximal element of $A_j$ and $x_2(\notin A_j) \ x_1$ then $f_j^1(x_2)=f_j^0(x_2)= f_j^1(x_2)\geq f_j^1(x_1)= f_j^1(x_1)$. So the maps $f_j^1$ are order preserving.

\textbf{Step II:} We now show that $f_j^1(gx) = f_{g(j)}^1(x)$ for any $j$ and any $x\in P, g\in \Sigma_n$. It is clear that $gx$ is a maximal element of $A_j$ if and only if $x$ is a maximal element of $g^{-1}A_j=A_{g(j)}$. If $gx$ is maximal element of $A_j$ then $f_j^1(gx) = f_j(gx) = f_{g(j)}(x)= f_{g(j)}^1(x)$. Also if $gx$ is not a maximal element of $A_j$ then $f_j^1(gx) = f_j^0(gx) = f_{g(j)}^0(x)= f_{g(j)}^1(x)$. So we have $f_j^1(gx) = f_{g(j)}^1(x)$ holds for any $j$ and  any $x\in P, g\in \Sigma_n$.

\textbf{Step III:}  Now we repeat this construction. We use $f_j^1$ and $f_j$ to define $f_j^2$, and use $f_j^2$ and $f_j$ to define $f_j^3$ and so on. By finiteness of $P$ and $Q$ this process will end and we get the order preserving maps  $f_j^0, f_j^1, f_j^2, \cdots, f_j^m=f_j:P \to Q$ such that $f_j^l(gx)=f_{g(j)}^l(x)$ for any $1\leq j \leq n, 0\leq l\leq m, g\in \Sigma_n$ and $x\in P$. 

\textbf{Step IV:} Now we show that the simplicial maps $\KK(f_1), \KK(f_2), \cdots, \KK(f_n) : \KK(P) \to \KK(Q)$ are symmetrically contiguous. Let $C= \{x_0\leq x_1\leq, \cdots,\leq x_c\}$ be a chain in $P$. Since the set $\{x\in P; f_j^{l-1}(x)\neq f_j(x)\}$ has at most one maximal element in $C$, so by definition of $f_j^l$ we have, $f_j^{l-1}$ and $f_j^l$ differ by at most one element on $C$ (say $x_p$). Thus we have $$f_j^{l-1}(x_0)=f_j^{l}(x_0)\leq  \cdots \leq f_j^{l-1}(x_p)\leq f_j^{l}(x_p) \leq \cdots \leq f_j^{l-1}(x_c)=f_j^{l}(x_c).$$ Therefore for any chain $C$ in P, $f_j^{l-1}(C) \cup f_j^l(C)$ is also a chain in $Q$. In other words for any simplex $\sigma$ of $\KK(P)$, $\KK(f_j^{l-1})(\sigma) \cup \KK(f_j^l)(\sigma)$ is a simplex in $\KK(Q)$. So $\KK(f_j^{l-1})(\sigma)$ and $\KK(f_j^l)(\sigma)$ are one contiguous. Also $f_j^l(gx) = f_{g(j)}^l(x)$ implies $\KK(f_j^l)(gx) = \KK(f_{g(j)}^l)(x)$ holds for any $x\in V(\KK(P))\simeq P$. So we can say that the simplicial maps $\KK(f_1), \KK(f_2), \cdots, \KK(f_n) : \KK(P) \to \KK(Q)$ are symmetrically contiguous.
\end{proof}

 To prove our main theorem of this section, we need the following lemma. It is a generalised version of \cite[Lemma 2.4]{Tan19}.
\begin{lemma}\label{contiguous to homotopy in face poset}
Let $K$ be a finite simplicial $\Sigma_n$-complex and the maps $\phi_1, \phi_2, \cdots, \phi_n : K \to L$ are symmetrically contiguous. Then the induced maps $\XX(\phi_j) : \XX(K) \to \XX(L)$ are symmetrically homotopic.
\end{lemma}

\begin{proof}
Assume that there is a chain of contiguous maps $\phi_*=\phi_j^0, \phi_j^1, \cdots, \phi_j^c=\phi_j : K \to L$ such that $\phi_*(gv)=\phi_*(v)$ and $\phi_j^l(gv) = \phi_{g(j)}(v)$ for $1\leq l \leq c, 1\leq j\leq n, v\in V(K)$ and $g\in \Sigma_n$. For each $l$ and $j$ we define 
$$h_j^l:\XX(K) \to \XX(L), ~~h_j^l(\sigma) = \phi_j^{l-1}(\sigma)\cup\phi_j^l(\sigma),$$
 for any $\sigma\in\XX(K)$. Since $\XX(\phi_j^{l-1})(\sigma) = \phi_j^{l-1}(\sigma)$ and $\XX(\phi_j^{l})(\sigma) = \phi_j^{l}(\sigma)$, so $\XX(\phi_j^{l-1}) \leq h_j^l \geq \XX(\phi_j^{l})$. We define a combinatorial homotopy $H : \XX(K) \times J_{n, 2c} \to \XX(L)$ by, 
$$
H(\sigma, m_j) =
 \begin{cases}
h_j^{\frac{m+1}{2}}(\sigma) & \mbox{ if } ~~ m  \mbox{ is odd,}\\
\XX\big(\phi_j^{\frac{m}{2}}\big)(\sigma) &\mbox{  if } ~~~~ m  \mbox{ is even}.
\end{cases}
$$ 
Now  $\XX(\phi_j^{l-1}) \leq h_j^l \geq \XX(\phi_j^{l})$ implies $H$ is order preserving and $H(\sigma, 2c_j) = \XX(\phi_j^{c})(\sigma)=\phi_j(\sigma)$. Also $$\XX(\phi_j^{l})(g\sigma) = \phi_j^{l}(g\sigma) = \phi_{g(j)}^{l}(\sigma) = \XX(\phi_{g(j)}^{l})(\sigma),$$ and 
$$h_j^l(g\sigma) = \phi_j^{l-1}(g\sigma)\cup\phi_j^l(g\sigma) = \phi_{g(j)}^{l-1}(\sigma)\cup\phi_{g(j)}^l(\sigma) = h_{g(j)}^l(\sigma).$$ Hence $H$ is a symmetric homotopy between $\XX(\phi_j), ~ 1 \leq j\leq n$.
\end{proof}
	
\end{mysubsection}

\begin{mysubsection}{Combinatorial complexity}
Here we recall the basics of higher combinatorial complexity $\ccn(P)$ as introduced in (\cite{Pau19}). Consider the mapping space $P^{J{n,m}}$, and the canonical order preserving map 
\begin{equation}\label{qnm}
 q_{n, m} : P^{J_{n, m}}
\rightarrow P^n, ~~q_{n,m}(\gamma) = (\gamma (m_1), \gamma(m_2),\cdots ,
\gamma (m_n)).
\end{equation}
 Note that Tanaka considered the following map to define $\CC_2(P)$,
$$q_m: P^{J_m} \rightarrow P \times P \text{ defined by } q_m(\gamma) = (\gamma (0), \gamma (m)), ~~m \geq 0.$$

\noindent Let $\tau_{P^n} : \sd(P^n) \rightarrow P^n$ be the map sending $p_0 \leq p_1 \leq \cdots \leq p_n$ to the last element $p_n$. This is a weak homotopy equivalence, and the induced simplicial map $\mathcal{K}(\tau_{P^n}) : \mathcal{K}(\sd(P^n)) = \sd(\mathcal{K}(P^n)) \rightarrow \mathcal{K}(P^n)$ is a simplicial approximation of the identity on $|\mathcal{K}(P^n)|$ (see \cite{HarV93}). For $r \geq 0$, we define
 \begin{equation}\label{tau}
 \tau_{P^n}^r : \sd^r(P^n) \rightarrow P^n
 \end{equation} 
as the composition
\begin{center}
$ \sd^r(P^n) \xrightarrow {\tau_{\sd^{r-1}(P^n)}} \sd^{r-1}(P^n) \xrightarrow {\tau_{\sd^{r-2}(P^n)}} \cdots   \xrightarrow {\tau_{\sd(P^n)}} \sd(P^n) \xrightarrow {\tau_{P^n}} P^n$.
\end{center}

\noindent Note that $\tau_{P^n}^0$ is identity map on $P^n$. For $r \geq 0$, $\CC_n^r(P)$ is defined to be the smallest non-negative integer $k$ such that there exist an open cover $\{Q_i\}_{i=1}^k$ of $\sd^r (P^n)$ and an positive integer $m$, with a map $s_i : Q_i \rightarrow P^{J_{n, m}}$ such that $q_{n, m} \circ s_{i} =  \tau_{P^n}^r$ on $Q_i$ for each $i$. If no such $k$ exists, then $\CC_n^r(P) = \infty$. The sequence $\{\CC_n^r(P)\}$ is a decreasing sequence on $r$. So we have the following definition for  \emph{$n$-th combinatorial complexity}.
	
 \begin{definition} 
 The \emph{$n$-th combinatorial complexity} $\ccn(P)$ of $P$ is defined as: 
 $$\ccn(P) = \min_{r \geq 0} \{\CC_n^r(P)\}. $$
\end{definition}
	
\noindent We have the following theorem.

\begin{theorem}[\cite{Pau19}]
For  any finite space $P$, we have $\ccn(P) = \tcn(|\mathcal{K}(P)|)$, $n \geq 2$. 
\end{theorem}

\end{mysubsection}
		
\begin{mysubsection}{Symmetric combinatorial complexity}
In (\cite{Tan19}), Tanaka defined symmetric combinatorial complexity $\ccs(P)$ for a finite space $P$. He considered the mapping space $P^{J_{2m}}$ and the canonical map $$q_{2m}: P^{J_{2m}} \rightarrow P \times P, ~~~ q_{2m}(\gamma) = (\gamma (0), \gamma (2m)), ~~m \geq 0.$$ 
 
\noindent Note that $P^{J_{2m}}$ and $P \times P$ both are $\zz$-spaces where the non trivial action is given by $\gamma \to \gamma^{-1}$ and $(x, y) \to (y, x)$. Also $q_{2m}$ is a $\zz$-map. Tanaka defined $\CC_{2,2m}^{\Sigma, r}(P)$ as the smallest non-negative integer $k$ such that there exist $\zz$-invariant open cover $\{Q_i\}_{i=1}^k$ of $\sd^r (P\times P)$ and on each $Q_i$ there is a $\zz$-map $s_i : Q_i \rightarrow P^{J_{2m}}$ with $q_{2m} \circ s_{i} =  \tau_{P\times P}^r$. If no such $k$ exists, then $\CC_{2,2m}^{\Sigma, r}(P) = \infty$. He proved that $\{\CC_{2,2m}^{\Sigma, r}(P)\}$ is a decreasing sequence on $m$ and $r$ both. He defined the \emph{symmetric combinatorial complexity} $\ccs(P)$ as the limit of $\CC_{2,2m}^{\Sigma, r}(P)$ as $m \to \infty, r\to \infty$.

Let us define $n$-th symmetric combinatorial complexity. Here we consider the space $P^{J{n,m}}$ and the map $q_{n, m} : P^{J_{n, m}} \rightarrow P^n$ as in Equation \ref{qnm}. Note that  ${J_{n, m}}$ is a $\Sigma_n$-space, the action given by $g(t_j)= t_{g(j)}, t\in\{1, 2, \cdots, m\},  j\in\{1, 2, \cdots, n\}$. This induces an action of $\Sigma_n$ on $P^{J{n,m}}$ given by $g\gamma(t_j)=\gamma(t_{g(j)})$. Also $P^n$ is a $\Sigma_n$-space, the action is given by $g(x_1, x_2, \cdots x_n) = (x_{g(1)}, x_{g(2)}, \cdots, x_{g(n)})$ and the maps $q_{n,m}$ in Equation \ref{qnm} and $\tau_{P^n}^r$ in Equation \ref{tau} are $\Sigma_n$-maps.
			
\begin{definition}
Let $P$ be a finite space. We define $\CC_{n,m}^{\Sigma, r}(P)$ as the minimum number $k$ such that there is a $\Sigma_n$-invariant open cover $\{Q_i\}_{i=1}^k$ of $\sd^r(P^n)$ and on each open set there is $\Sigma_n$-map $s_i : Q_i \to P^{J_{n,m}}$ satisfying $q_{n, m} \circ s_{i} =  \tau_{P^n}^r$. If no such $k$ exists, then we define $\CC_{n,m}^{\Sigma, r}(P)$ to be $\infty$.
\end{definition}

\begin{lemma}\label{lemma ccm}
Let $P$ be a finite space. Then $\CC_{n,m}^{\Sigma, r}(P) \geq \CC_{n,m+1}^{\Sigma, r}(P)$, for any $m,r\geq0$ and $n\geq 2$.
\end{lemma}

\begin{proof}

 We assume that $\CC_{n,m}^{\Sigma, r}(P) = k$ and  $\{Q_i\}_{i=1}^k$ is a $\Sigma_n$-invariant open cover of $\sd^r(P^n)$ such that  on each open set there is $\Sigma_n$-map $s_i : Q_i \to P^{J_{n,m}}$ satisfying $q_{n, m} \circ s_{i} =  \tau_{P^n}^r$. Consider the retraction map $R : J_{n, m+1} \rightarrow J_{n, m}$ sending each $(m+1)_j$ to $m_j$ for $1\leq j \leq n$. Clearly this is an order preserving map and it induces a $\Sigma_n$-map $R^* : P^{J_{n, m}} \rightarrow P^{J_{n, m+1}}, \gamma \rightarrow \gamma \circ R$. Then we have the following commutative diagram.
  $$
  \xymatrix{
 P^{J_{n, m}} \ar[rrrr]^{R^{*}}  \ar[ddrr]_{q_{n, m}}& &&&P^{J_{n, m+1}} \ar[ddll]^{q_{n, m+1}}\\
 && Q_i\subset\sd^r(P^n) \ar[d]^{\tau_{P^n}^r} \ar[llu]_{s_i} \ar@{.>}[rru]^{s'_{i}} \\ && P^n 
 }
 $$

\noindent Note that the composition map $R^* \circ s_i=s'_i : Q_i \rightarrow P^{{n, m+1}}$ is a $\Sigma_n$-map and holds $q_{n, m+1} \circ s'_{i} =  \tau_{P^n}^r$ for each $i$. Thus, $\CC_{n,m+1}^{\Sigma, r}(P) \leq k$.

\end{proof}

\begin{remark} Since for any $m$ the value $\CC_{n,m}^{\Sigma, r}(P)\geq 1$, so by Lemma \ref{lemma ccm} the value is stable for large $m$. We denote the stable value by $\ccnsr(P)$. 
\end{remark}

For $r \geq 0$, let $\rho_j : \sd^r(P^n) \rightarrow P$ denote the composition of $\tau^r_{P^n} : \sd^r(P^n) \rightarrow P^n$ and the $j$-th projection  $p_j:P^n \to P$ for $j = 1, 2,\cdots,n$. We have the following lemma which gives an alternative formulation of the definition for $\ccnsr(P)$.

\begin{lemma}\label{lemma cc homotopy}
With notations as above, $\ccnsr(P)$ is the minimal number $k$ such that there exist a $\Sigma_n$-invariant open cover $\{Q_i\}_{i = 1}^k$ of $\sd^r(P^n)$ and the maps $\rho_1, \rho_2,\cdots,\rho_n : Q_i \rightarrow P$ are symmetrically homotopic.
\end{lemma}

\begin{proof}
The existence of maps $s : Q \to P^{J_{n,m}}$ and $H : Q \times J_{n, m} \to P$ are equivalent by exponential law. Let $ \bx = (x_1,  x_2, \cdots, x_n)\in Q$. We set $s(\bx)(t_j) =  H(\bx, t_j),$ for $ t_j$ is parameter for $j$-th fence and $j\in \{1, 2, \cdots, n\}$. If $Q$ is an $\Sigma_n$-invariant set then $s$ is an $\Sigma_n$-map if and only if $H$ satisfy the relation $H(g\bx, t_j) = H(\bx, t_{g(j)})$.
Also we have,
\begin{align*}
q_{n, m} \circ s_{i} =  \tau_{P^n}^k & \Leftrightarrow [q_{n, m} \circ s_{i}(\bx)](j) =  [\tau_{P^n}^k(\bx)](j) \\ &\Leftrightarrow s(\bx)(m_j) = \rho_j(\bx) \\
&\Leftrightarrow H(\bx, m_j) = \rho_j(\bx) \\ & \Leftrightarrow H \text{  is a homotopy between } \rho_1, \rho_2,\cdots,\rho_n.
\end{align*}
Hence the Lemma follows.
\end{proof}
The value $\ccnsr(P)$ depends on $r$. If we increase $r$ the value $\ccnsr(P)$ will decrease. We have the following Proposition for $r=0$. This is a symmetric version of \cite[Theorem 4.9]{Pau19}.

\begin{prop}
For any finite space $P$,  we have $\CC_n^{\Sigma, 0}(P) = \TC_n^{\Sigma}(P)$, $n\geq 2$. 
\end{prop}

\begin{proof}
	
Using Lemma \ref{lemma cc homotopy} we get, $\CC_n^{\Sigma, 0}(P)$ is the minimal integer $k$ such that there exist a $\Sigma_n$-invariant open cover $\{Q_i\}_{i = 1}^k$ of $P^n$ and the projection maps $\rho_1, \rho_2,\cdots,\rho_n : Q_i \rightarrow P$ are symmetrically homotopic if and only if $\TC_n^{\Sigma}(P) = k$ (using Lemma \ref{lemma comb homotopy} and Remark \ref{remark sym tc definition}).

\end{proof}

\begin{lemma}\label{cc geq sc}
Let $P$ be a finite space, and $r\geq0$ and $n\geq 2$. Then we have,
\begin{enumerate}[(i)]
\item $\CC_{n}^{\Sigma, r}(P) \geq \SC_{n}^{\Sigma}(\KK(P))$,
\item $\CC_{n}^{\Sigma, r}(P) \geq \CC_{n}^{\Sigma, r+1}(P)$.
\end{enumerate}
\end{lemma}

\begin{proof}

(i) Assume that $\CC_{n}^{\Sigma, r}(P)  = k$. By Lemma \ref{lemma cc homotopy} there exists a symmetric open cover $\{Q_i\}_{i=1}^{r}$ of $\sd^r(P^n)$ such that $\rho_1, \rho_2,\cdots,\rho_n : Q_i \rightarrow P$ are symmetrically homotopic. By Proposition \ref{homotopy to contiguous in face poset} we can say that the maps 
$\mathcal{K}(\rho_1), \mathcal{K}(\rho_2),\cdots,\mathcal{K}(\rho_n) : \mathcal{K}(Q_i) \rightarrow \mathcal{K}(P)$
are symmetrically contiguous. Since $Q_i$ is symmetric open set implies the subcomplex $\mathcal{K}(Q_i)$ is symmetric and $\{\mathcal{K}(Q_i)\}_{i=1}^k$ is a cover of $\mathcal{K}(\sd^r(P^n)) = \sd^r(\mathcal{K}(P^n))$. Also $\mathcal{K}(\rho_j) = \mathcal{K}(p_j \circ \tau^r_{P^n}) = \KK(p_j)\circ \iota^r_{\KK(P^n)}= \pi_j$, for $j = 1, 2,\cdots,n$ and by Lemma \ref{lemma approx of identity} $\iota^r_{\KK(P^n)}$ is $\Sigma_n$-approximation of identity. So, $\scnsr(\KK(P)) \leq k$ and therefore $\SC_{n}^{\Sigma}(\KK(P)) \leq k$.
	
(ii)  Let $\CC_{n}^{\Sigma, r}(P) = k$. Then we have a $\Sigma_n$-invariant open cover $\{Q_i\}_{i=1}^{k}$ of $\sd^{r}(P^n)$ such that on each $Q_i$ there is a $\Sigma_n$-map $s_i : Q_i \to P^{J_{n, m}}$ satisfying $q_{n, m} \circ s_{i} =  \tau_{P^n}^r$ for some $m \geq 0$. Let us take the open cover $\{U_i\}_{i=1}^{k}$ of $\sd^{r+1}(P^n)$, where $U_i = \tau^{-1}_{\sd^{r}(P^n)}(Q_i)$ and set $s'_i = s_i \circ \tau_{\sd^{r}(P^n)} : U_i \to P^{J_{n, m}}$ such that the following diagram commutes:
$$
\xymatrix{
U_i \ar@{^{(}->}[dd] \ar[rr]^{\tau_{\sd^{r}(P^n)}} \ar@/^{2.5pc}/@{..>}[rrrr]\sp{s'_i}  &&  Q_i \ar[rr]^{s_i} \ar@{^{(}->}[dd]  &&  P^{J_{n, m}}  \ar[dd]^{q_{n, m}} \\\\
\sd^{r+1}(P^n)  \ar[rr]_{\tau_{\sd^{r}(P^n)}} && \sd^{r}(P^n) \ar[rr]_{\tau^r_{P^n}} && P^n}
$$
Here $Q_i$ is $\Sigma_n$-invariant implies $U_i$ so and $s_i$ is a $\Sigma_n$-map implies $s'_i$ also. So we get a $\Sigma_n$-invariant open cover $\{U_i\}_{i=1}^{k}$ of $\sd^{r+1}(P^n)$  and a $\Sigma_n$-map $s'_i$ on each $U_i$ such that $q_{n, m} \circ s'_{i} =  \tau_{P^n}^{r+1}$. Therefore $\CC_{n}^{\Sigma, r+1}(P) \leq k$.

\end{proof}

\noindent We get $\{\CC_{n}^{\Sigma, r}(P)\}_r$ is a decreasing sequence bounded below by $1$. So we make the following definition.
\begin{definition} 
For a finite space $P$, the \emph{$n$-th symmetric combinatorial complexity} $\ccns(P)$ of $P$ is defined as: 
 $$\ccns(P) = \min_{r \geq 0} \{\ccnsr(P)\}. $$
\end{definition}
Let us now prove our main theorem. This can be viewed as a symmetric version of \cite[Theorem 5.6]{Pau19} and higher analogue of \cite[Theorem 3.15]{Tan19}.
\begin{theorem}\label{thm sym cc sym tc equal}
For a finite space $P$, $\ccns(P) = \scns(\KK(P)) = \tcns(|\KK(P)|)$, for any $n\geq2$.
\end{theorem}
\begin{proof}
In (i) of Lemma \ref{cc geq sc}, if we take $r \to \infty$, we get $\ccns(P) \geq \scns(\KK(P))$. Let us prove $\ccns(P) \leq \scns(\KK(P))$. Let $\scns(\KK(P))=k$ and $\{L_i\}_{i=1}^k$ be a collection of symmetric subcomplexes covering $\sd^r(\mathcal{K}(P^n))$ such that the restrictions $\pi_1, \pi_2,\cdots,\pi_n : L_i \rightarrow \mathcal{K}(P)$ are symmetrically contiguous, for some $r\geq0$ and for each $i$. By Lemma \ref{contiguous to homotopy in face poset} the maps $\mathcal{X}(\pi_1), \mathcal{X}(\pi_2),\cdots,\mathcal{X}(\pi_n) : \mathcal{X}(L_i) \rightarrow \mathcal{X}(\mathcal{K}(P)) = \sd(P)$ are symmetrically homotopic. So the composition maps $\tau_P\circ\mathcal{X}(\pi_1), \tau_P\circ\mathcal{X}(\pi_2), \cdots, \tau_P\circ\mathcal{X}(\pi_n) : \mathcal{X}(L_i) \rightarrow P$ are also symmetrically homotopic. Using Remark \ref{face poset of a subcomplex open}, since $L_i$ is a subcomplex of $\sd^r(\mathcal{K}(P^n)),$ so $\XX(L_i)$ is open in $\mathcal{X}(\sd^r(\mathcal{K}(P^n))) = \sd^{r+1}(P^n)$, for each $i$. Now $\{L_i\}_{i=1}^k$ is a collection of symmetric subcomplexes covering $\sd^r(\mathcal{K}(P^n))$ implies $\{\mathcal{X}(L_i)\}_{i=1}^k$ is a symmetric open cover of $\sd^{r+1}(P^n)$. Also for any $j\in\{1, 2, \cdots, n\}$ we have,
\begin{align*}
\tau_P \circ \mathcal{X}(\pi_j) &= \tau_P \circ \mathcal{X}(\mathcal{K}(p_j \circ \tau^k_{P^n}))  = \tau_P \circ \sd(p_j \circ \tau^k_{P^n})\\
&=  \tau_P \circ \sd(p_j) \circ \sd(\tau^k_{P^n})  = p_j \circ \tau_{P^n}^{k+1}  = \rho_j,
\end{align*}
where $p_j:P^n \to P$ is the $j$-th projection. Hence $\ccns(P) \leq k = \scns(\KK(P))$ and using Theorem \ref{thm sym sc sym tc equal} we have $\ccns(P) = \scns(\KK(P)) = \tcns(|\KK(P)|)$, for any $n\geq2$.
\end{proof}
\end{mysubsection}
\end{section}


\begin{thebibliography}{12}
	
\bibitem{BarM12} J.~Barmak \& E.G.~Minian,
``Strong homotopy types, nerves and collapses," (English summary) \textit{Discrete Comput. Geom.} \textbf{47} (2012), no. 2, pp.~301 - 328.
		
\bibitem{BasGRT14} I.~Basabe, J.~ Gonz\'{a}lez, Y. B.~Rudyak \& D.~Tamaki,
``Higher topological complexity and its symmetrization," \textit{Algebr. Geom. Topol.}  \textbf{14} (2014), no. 4, pp.~2103 - 2124.

\bibitem{Bre72}  G. E. ~Bredon,
\textit{Introduction to compact transformation groups,} Pure and Applied Mathematics, \textbf{46}, Academic Press, New York-London, 1972, xiii+459.
		
\bibitem{EilS52} S.~Eilenberg \& N.~Steenrod,
``Foundations of algebraic topology," \textit{Princeton University Press, Princeton, New Jersey} (1952), xv+328.
					
\bibitem{Far08} M.~Farber,
\textit{Invitation to topological robotics, Zurich Lectures in Advanced Mathematics,}
European Mathematical Society (EMS), Zürich, (2008) x+133.
		
\bibitem{Far03} M.~ Farber,
``Topological complexity of motion planning," \textit{Discrete Comput. Geom.}
\textbf{29} (2003), no. 2, pp.~211 - 221.
		
\bibitem{FarG07} M.~ Farber and M.~ Grant,
``Symmetric motion planning," \textit{Contemp. Math.} \textbf{438}, \textit{Amer. Math. Soc.} (2007), pp.~85 - 104.
		
\bibitem{Gon18} J.~ Gonz\'{a}lez,
``Simplicial complexity: piecewise linear motion planning in robotics," \textit{New York J. Math.}\textbf{ 24} (2018), pp.~27 - 292.
		
\bibitem{HarV93} K.A.~Hardie and J.J.C~ Vermeulen,
``Homotopy theory of finite and locally finite $T_0$ spaces," \textit{Exposition. Math.} \textbf{11} (1993), pp.~ 331 - 341.
		
\bibitem{Pau19} A. K.~Paul,
`` Higher analogs of simplicial and combinatorial complexity," \textit{Topology Appl.} \textbf{267} (2019).
		
\bibitem{PorSZ17} M.~Pors, S.~ Sarkar \& P.~Zvengrowski,
`` Remarks about {$\Delta$}-complexes and applications," \textit{Homology Homotopy Appl.} \textbf{19} (2017), no. 1, pp.~89 - 110.
		
\bibitem{Pra07}  V. V.~Prasolov,
\textit{Elements of homology theory,} Graduate Studies in Mathematics \textbf{81}, American Mathematical Society, Providence, RI, 2007, x+418.
		
\bibitem{Rud10} Y. B. ~Rudyak,
`` On higher analogs of topological complexity," \textit{Topology Appl.} \textbf{157} (2010), no. 5, pp.~916 - 920.

\bibitem{Spa66}  E. H.~Spanier,
\textit{Algebraic topology,} McGraw-Hill Book Co., New York-Toronto, Ont.-London, 1966, xiv+528.
		
\bibitem{Sto66} R.E.~Stong,
``Finite topological spaces," \textit{Trans. Amer. Math. Soc.} \textbf{123} 1966 pp.~325 - 340.
		
\bibitem{Tan18} K.~Tanaka,
``A combinatorial description of topological complexity for finite spaces," (English summary) \textit{Algebr. Geom. Topol.} \textbf{18} (2018), no. 2, pp.~779 - 796.
		
\bibitem{Tan19} K.~Tanaka,
``Symmetric topological complexity for finite spaces and classifying spaces," \textit{Topol. Methods Nonlinear Anal.} \textbf{54} (2019), no. 2, pp.~477 - 493.

		
\end{thebibliography}
\end{document}